\documentclass{amsart}

\usepackage{amsmath, amssymb, amsthm,enumitem,bbm,bm,pict2e,xcolor}
\usepackage{indentfirst}

\theoremstyle{definition}
\newtheorem{theorem}{Theorem}
\numberwithin{theorem}{section}
\newtheorem{proposition}[theorem]{Proposition}
\newtheorem{lemma}[theorem]{Lemma}
\newtheorem{remark}[theorem]{Remark}

\newtheorem{cor}[theorem]{Corollary}

\DeclareMathOperator{\ord}{ord}
\DeclareMathOperator{\grad}{grad}
\DeclareMathOperator{\spanned}{span}

\newcommand{\la}{\lambda}

\title{FFLV-type monomial bases for type $B$}

\author{I. Makhlin}
\address{Igor Makhlin:\newline
Skolkovo Institute of Science and Technology, Nobelya Ulitsa 3, Moscow 121205, Russia
\newline
{\it and }\newline
National Research University Higher School of Economics, 
International Laboratory of Representation Theory and Mathematical Physics,
Usacheva str. 6, 119048, Moscow, Russia}
\email{imakhlin@mail.ru}

\date{}

\begin{document}
\begin{abstract}
We present a combinatorial monomial basis (or, more precisely, a family of monomial bases) in every finite-dimensional irreducible $\mathfrak{so}_{2n+1}$-module. These bases are in many ways similar to the FFLV bases for types $A$ and $C$. They are also defined combinatorially via sums over Dyck paths in certain triangular grids. Our sums, however, involve weights depending on the length of the corresponding root. Accordingly, our bases also induce bases in certain degenerations of the modules but these degenerations are obtained not from the filtration by PBW degree but by a weighted version thereof.
\end{abstract}

\maketitle

\section*{Introduction}

In the papers~\cite{FFL1} and~\cite{FFL2} Feigin, Fourier and Littelmann constructed certain monomial bases in the finite-dimensional irreducible representations of, respectively, type $A$ and type $C$ simple Lie algebras. These bases came to be known as the {\it FFLV bases}, with ``FFL'' being the initials of the three authors and the ``V'' standing for Vinberg, who was the first to conjecture the result for type $A$ in~\cite{V}. 

Here we use the word ``monomial'' to denote the fact that each of the basis vectors is obtained from the highest vector by the action of a monomial in the root vectors. The degrees of these monomials are given by integer points in certain polytopes (FFLV polytopes). Thus these bases comprise a fascinating and relatively new family of combinatorial bases entirely different from the classic Gelfand-Tsetlin bases~(\cite{Mo}). FFLV bases serve as a key component of the growing theory of PBW degenerations. This theory reaches into various aspects of representation theory (\cite{Fe1}, \cite{FFL1}, \cite{FFL2}, \cite{CF}, ...), algebraic geometry (\cite{Fe2}, \cite{CFR}, \cite{H}, \cite{Ki}, ...) and combinatorics (\cite{ABS}, \cite{K}, \cite{Fo}, \cite{FM}, ...). 

That being said, versions of FFLV bases for the remaining (i.e. orthogonal) classical Lie algebras have yet to be constructed. In this paper we offer a possible solution for type $B$. (We point out that constructions for certain special cases in type $B$ can be found in~\cite{BK}. Those constructions are not a special case of the ones presented here.)

Parallels between the bases constructed in this paper and FFLV bases for types $A$ and $C$ can be drawn on two levels: combinatorial and algebraic. 

The combinatorial definition of our bases is remarkably similar to that of FFLV bases: the roots of the type $B$ root system are arranged into a triangular grid and the degrees of the monomials defining our bases are obtained by limiting sums over ``Dyck paths'' in the grid. A key difference is that one computes these sums with weights depending on the length of the root (i.e. short roots have a weight of $\frac12$) which is not the case for type $C$. It is also worth mentioning that, unlike both types $A$ and $C$, the way in which we arrange the roots differs slightly from the Hasse diagram of their standard ordering.

On the algebraic level our bases induce bases in certain associated graded spaces (degenerations) of the representation, as is the case for FFLV bases. These degenerations are again defined by a filtration which is obtained by computing certain degrees for every PBW monomial. However, the degree we consider here is not the regular PBW degree but a weighted modification thereof. Short roots contribute a summand of $\frac12$ to the degree which is seen to reflect the above difference on the combinatorial level. We point out that the associated graded algebra for such a filtration is not commutative unlike the standard PBW filtration. Therefore, we may not assume that a basis is obtained regardless of the order in which the root vectors in every monomial are found and, in fact, not all orders provide a basis (see also Remark~\ref{remAC}).

\section{Definitions and the main result}\label{defsec}

Consider the complex Lie algebra $\mathfrak{g}=\mathfrak{so}_{2n+1}$. Fix a Cartan decomposition $\mathfrak g=\mathfrak n^-\oplus\mathfrak h\oplus\mathfrak n^+$. We choose a basis $\beta_1,\ldots,\beta_n$ in $\mathfrak h^*$ such that the set $\Phi^+$ of positive roots consists of the vectors $\beta_i-\beta_j$ for $1\le i<j\le n$, $\beta_i+\beta_j$ for $1\le i<j\le n$ and $\beta_i$ for $1\le i\le n$. The basis $(\beta_i)$ is orthonormal with respect to the (dual of the) Killing form. The simple roots are then the vectors $\alpha_i=\beta_i-\beta_{i+1}$ for $1\le i\le n-1$ together with $\alpha_n=\beta_n$. The fundamental weights are the vectors $\omega_i=\beta_1+\ldots+\beta_i$ for $1\le i\le n-1$ together with $\omega_n=\frac12(\beta_1+\ldots+\beta_n)$. This information concerning root systems of type $B$ can be found, for instance, in~\cite{carter}.

Fix a dominant integral weight $\lambda\in\mathfrak h^*$. Let $\lambda$ have coordinates $(a_1,\ldots,a_n)$ with respect to the basis of fundamental weights (the $a_i$ being arbitrary nonnegative integers) and coordinates $(\la_1,\ldots,\la_n)$ with respect to the basis $(\beta_i)$. We then have the relations $$\lambda_i=\sum_{j=i}^{n-1}a_j+\frac{a_n}2,$$ wherefrom we see that the coordinates $(\la_i)$ comprise a non-increasing sequence of nonnegative half-integers, pairwise congruent modulo 1.

We now move on to define the combinatorial set $\Pi_\lambda$ which parametrizes our basis (or, rather, each of a family of bases) in the irreducible representation $L_\la$ with highest weight $\la$. Each element of $\Pi_\la$ is a number triangle consisting of $n^2$ nonnegative integers  $T_{i,j}$ with $1\le i<j \le 2n+1-i$. We visualize these triangles with $T_{i,j}$ and $T_{i+1,j+1}$ being, respectively, the upper-left and the upper-right neighbors of $T_{i,j+1}$, e.g. for $n=3$ we have:
\begin{center}
\begin{tabular}{ccccc}
$T_{1,2}$ && $T_{2,3}$ && $T_{3,4}$\\
&$T_{1,3}$ && $T_{2,4}$ & \\
&& $T_{1,4}$ && $T_{2,5}$ \\
&&& $T_{1,5}$ &\\
&&&& $T_{1,6}$
\end{tabular}
\end{center}
\vspace{2ex}

To specify when $T\in\Pi_\lambda$ the notion of a {\it Dyck path} is used. We call a sequence $$d=((i_1,j_1),\ldots,(i_N,j_N))$$ of pairs $1\le i<j \le 2n+1-i$ a Dyck path if we have $j_1-i_1=1$, the element $(i_{k+1},j_{k+1})$ is either $(i_k+1,j)$ or $(i_k,j_k+1)$ for all $1\le k\le N-1$ and, lastly, either $j_N-i_N=1$ or $i_N+j_N=2n+1$. In terms of the above visualization this means that the path starts in the top horizontal row, every element of the path is either the bottom-right or the upper-right neighbor of the previous one and that the path ends either in the top row or in the rightmost vertical column.

For a triangle $T=(T_{i,j}, 1\le i<j \le 2n+1-i)$ and a Dyck path $$d=((i_1,j_1),\ldots,(i_N,j_N))$$ we denote 
$$S(T,d)=
\begin{cases}
\sum\limits_{(i,j)\in d} T_{i,j}&\text{ if }i_N+j_N<2n+1,\\
\sum\limits_{l=1}^{N-1} T_{i,j}+\frac{T_{i_N,j_N}}2&\text{ if }i_N+j_N=2n+1.
\end{cases}
$$ 
In the above dichotomy we distinguish between the paths ending in the top row but not the rightmost column and those that do end in the rightmost column.
Next, we define
$$M(\lambda,d)=
\begin{cases}
\la_{i_1}-\la_{j_N}&\text{ if }i_N+j_N<2n+1,\\
\la_{i_1}&\text{ if }i_N+j_N=2n+1.
\end{cases}
$$ 
We now define $\Pi_\lambda$ by saying that $T\in\Pi_\lambda$ if and only if all $T_{i,j}$ are nonnegative integers and for any Dyck path $d$ we have $S(T,d)\le M(\lambda,d)$. 

Next, let us establish a one-to-one correspondence between the elements of a triangle $T\in\Pi_\la$ and the positive $\mathfrak g$-roots. Namely, let the root $\alpha_{i,j}$ corresponding to $T_{i,j}$ be the root $\beta_i-\beta_j$ when $j\le n$, the root $\beta_i+\beta_{2n+1-j}$ when $n<j<2n+1-i$ and the short root $\beta_i$ when $i+j=2n+1$. 

(To elaborate on the remark in the introduction we point out that in the Hasse diagram of the standard order on the set of positive roots we would have the short roots positioned on the ``diagonal'' of our triangle, i.e. in the positions $(i,n+1)$, instead of the rightmost vertical column.)

For every pair $1\le i<j\le 2n+1-i$ fix a nonzero element $f_{i,j}\in\mathfrak n^-$ in the root space of $-\alpha_{i,j}$. We assume our choice of the $f_{i,j}$ to be standard in the sense that the Serre relations are satisfied. In what follows we will often write $\overline j$ for $2n+1-j$ to simplify our notations. This means that for every pair $1\le i<j\le n$ we have the positive roots $\alpha_{i,j}=\beta_i-\beta_j$ and $\alpha_{i,\overline j}=\beta_i+\beta_j$ and for every $1\le i\le n$ we have the short root $\alpha_{i,\overline i}=\beta_i$. We now also give the only commutation relation that we will be using explicitly:
\begin{equation}\label{comrel}
[f_{i,\overline i},f_{j,\overline j}]=2f_{i,\overline j}\text{ for }1\le i<j\le n.
\end{equation}

We introduce a few more concepts before stating our main theorem. First, for a monomial $M\in\mathcal U(\mathfrak n^-)$ in the root vectors $f_{i,j}$ let $\log M$ denote the number triangle $T_{i,j}$, $1\le i<j\le 2n+1-i$ where $T_{i,j}$ is equal to the total degree in which $M$ contains $f_{i,j}$. Second, let us call such an $M$ {\it arranged} if for all pairs $1\le i<j\le n$ the monomial $M$ does not contain an $f_{j,\overline j}$ to the left of an $f_{i,\overline i}$ (i.e. the $f_{i,\overline i}$ occurring in $M$ are ordered by $i$ increasing from left to right).

Finaly, denote $v_0$ a highest weight vector in $L_\la$.
\begin{theorem}\label{main}
For every triangle $T\in \Pi_\la$ choose an arranged monomial $M_T\in\mathcal U(\mathfrak n^-)$ with $\log M_T=T$ and denote $v_T=M_T (v_0)\in L_\la$. For any such choice the resulting set $\{v_T, T\in\Pi_\la\}$ will constitute a basis in $L_\la$.
\end{theorem}

\begin{remark}\label{remAC}
We see that, unlike the FFLV bases for types $A$ and $C$ constructed in~\cite{FFL1} and~\cite{FFL2}, the root vectors in the monomials may not be ordered arbitrarily. Instead we require the monomials to be arranged. It can be easily seen that some restriction of this sort must indeed be imposed. For instance, for $n=2$ and $\lambda=2\omega_2$ the module $L_\la$ may be described as $\wedge^2 V$ with $V=\spanned(e_2,e_1,e_0,e_{-1},e_{-2})$ and highest weight vector $e_1\wedge e_2$ (see Section~\ref{basesec} for a detailed description of the representation). Here one already has $\log(f_{2,3}f_{1,4}f_{2,3})\in\Pi_\lambda$ but \[f_{2,3}f_{1,4}f_{2,3}(v_0)=f_{2,3}f_{1,4}f_{2,3}(e_1\wedge e_2)=f_{2,3}f_{1,4}(e_1\wedge e_0)=-2f_{2,3}(e_1\wedge e_{-1})=0.\]

However, numerical experimentation has shown that a basis may sometimes be obtained despite some of the $M_T$ not being arranged. For example, it seems plausible that avoiding monomials with a subexpression of the form $f_{i,\overline i}\ldots f_{j,\overline j}\ldots f_{i,\overline i}$ with $i\neq j$ is sufficient to obtain a basis. Still, even such a theorem would not be the most general in this vein. It would be interesting to know if there exists a concise combinatorial criterion distinguishing those sets of monomials $\{M_T, \log M_T=T, T\in\Pi_\la\}$ for which $\{v_T=M_T(v_0)\}$ is a basis in $L_\lambda$. 
\end{remark}

\section{Bijection with the Gelfand-Tsetlin basis}

Our first step towards the proof of Theorem~\ref{main} will be showing that we indeed have $|\Pi_\la|=\dim L_\la$. This will be done by establishing a bijection between $\Pi_\la$ and the Gelfand-Tsetlin basis in $L_\la$ (see~\cite{Mo}) or, more precisely, the corresponding set of Gelfand-Tsetlin patterns.

The set $\Gamma_\la$ of Gelfand-Tsetlin patterns parametrizing the Gelfand-Tsetlin basis in $L_\la$ consists, once again, of certain number triangles $R=\{R_{i,j}\}$ with $1\le i<j\le \overline i$. We have $R\in\Gamma_\la$ if and only if the following requirements are met. 
\begin{enumerate}
\item All $R_{i,j}$ are nonnegative half-integers and if $i+j<2n+1$, then $R_{i,j}$ is congruent to (all of) the $\lambda_i$ modulo 1.
\item For $1\le i\le n-1$ we have $\lambda_i\ge R_{i,i+1}\ge\la_{i+1}$ and we also have $\la_n\ge R_{n,n+1}$.
\item If $j-i>1$ and $i+j<2n+1$, then $R_{i,j-1}\ge R_{i,j}\ge R_{i+1,j}$. If $j-i>1$ and $i+j=2n+1$, then $R_{i,j-1}\ge R_{i,j}$.
\end{enumerate}

When considering a pattern $R\in\Gamma_\la$ we at times refer to $n$ additional fixed elements $R_{i,i}=\la_i$ for $1\le i\le n$. These are naturally visualized as an additional top row of the triangle. Then (2) and (3) simply state that every element is no greater than its upper-left neighbor and no less than its upper-right neighbor (whenever the neighbor in question exists).

In~\cite{Mo} it is shown that the defined set $\Gamma_\la$ parametrizes a certain basis in $L_\la$, we now define a map $F:\Gamma_\la\to\Pi_\la$ which we then show to be bijective. Namely, for a pattern $R\in\Gamma_\la$ and a pair $1\le i<j\le\overline i$ set 
$$
F(R)_{i,j}=
\begin{cases}
\min(R_{i,j-1},R_{i-1,j})-R_{i,j}&\text{ if }i>1\text{ and }i+j<2n+1,\\
R_{i,j-1}-R_{i,j}&\text{ if }i=1\text{ and }j<2n,\\
2(\min(R_{i,j-1},R_{i-1,j})-R_{i,j})&\text{ if }i>1\text{ and }i+j=2n+1,\\
2(R_{i,j-1}-R_{i,j})&\text{ if }i=1\text{ and }j=2n.
\end{cases}
$$ 
Note that $R_{i,j-1}$ and $R_{i-1,j}$ are, respectively, the upper-left and bottom-left neighbors of $R_{i,j}$. Above, one of the first two cases takes place if $R_{i,j}$ is not in the rightmost vertical column, otherwise, one of the last two takes place. Cases 1 and 3 take place when $R_{i,j}$ does have a bottom-left neighbor, otherwise, one of cases 2 and 4 takes place (i.e. $i=1$).

\begin{theorem}
The map $F:\Gamma_\la\to\Pi_\la$ is well-defined and bijective.
\end{theorem}
\begin{proof}
First we show that the image of $F$ is contained in $\Pi_\lambda$. The fact that all $F(R)_{i,j}$ for an $R\in\Gamma_\la$ are nonnegative integers is immediate from the definition of $F$ and properties (1) - (3) above. Now consider a Dyck path $d=((i_1,j_1),\ldots,(i_N,j_N))$. If $j_N-i_N=1$ and $i_N+j_N<2n+1$, we have 
\begin{multline*}
S(F(R),d)\le (\la_{i_1}-R_{i_1,j_1})+(R_{i_1,j_1}-R_{i_2,j_2})+\ldots+(R_{i_{N-1},j_{N-1}}-R_{i_N,j_N})\le\\\la_{i_1}-R_{i_N,j_N}\le\la_{i_1}-\la_{j_N}.
\end{multline*}
If $i_N+j_N=2n+1$, we have
\begin{multline*}
S(F(R),d)\le (\la_{i_1}-R_{i_1,j_1})+(R_{i_1,j_1}-R_{i_2,j_2})+\ldots+(R_{i_{N-1},j_{N-1}}-R_{i_N,j_N})\le\\\la_{i_1}-R_{i_N,j_N}\le\la_{i_1}.
\end{multline*}

To prove that $F$ is bijective we describe the inverse map $G:\Pi_\la\to\Gamma_\la$. 

First we introduce the notion of a {\it partial Dyck path}. A partial Dyck path is a sequence $d=((i_1,j_1),\ldots,(i_N,j_N))$ such that for all $1\le k\le N$ we have $1\le i_k<j_k\le \overline{i_k}$, that $j_1-i_1=1$ and that for all $1\le k\le N-1$ the pair $(i_{k+1},j_{k+1})$ is equal to either $(i_k+1,j_k)$ or $(i_k,j_k+1)$.  For such a $d$ and a number triangle $T=\{T_{i,j},1\le i<j\le \overline i\}$ we set $$S(T,d)=\sum_{k=1}^N c_k T_{i_k,j_k},$$ where $c_k=1$ if $i_k+j_k<2n+1$ and $c_k=\frac12$ otherwise. 

For a $T\in\Pi_\la$ we define $$G(T)_{i,j}=\min\limits_{\substack{\text{partial Dyck path}\\d=((i_1,j_1),\ldots,(i,j))}}(\la_{i_1}-S(T,d)).$$ Let us show that $G(T)\in\Gamma_\la$. Property (1) is immediate except for the nonnegativity. Any partial Dyck path $d=((i_1,j_1),\ldots,(i,j))$ may be extended to a Dyck path $d'=((i_1,j_1),\ldots,(i',j'))$ with $i'+j'=2n+1$. However, $S(T,d)\le S(T,d')\le\la_{i_1}$ and the nonnegativity ensues. Property (3) is immediate from the definition of $G$, the fact that $G(T)_{i,i+1}\le\la_i$ for $1\le i\le n$ is also immediate. Now, if for some $1\le i\le n-1$ we had $G(T)_{i,i+1}<\la_{i+1}$, then for some Dyck path $d=((l,l+1),\ldots,(i,i+1))$ we would have $S(T,d)=\la_l-G(T)_{i,i+1}>\la_l-\la_{i+1}=M(\la,d)$.

Finally, the fact that $F$ and $G$ are mutually inverse is straightforward from their definitions.
\end{proof}

\begin{cor}\label{gt}
$|\Pi_\la|=\dim L_\la$.
\end{cor}

Since the above equality has been established it suffices to either prove that the set considered in Theorem~\ref{main} spans $L_\lambda$ or that it is linearly independent. In fact, we will proceed by a certain induction on $\la$ and, in a sense, prove the former for the base and the latter for the step.

\begin{remark}\label{polytope}
It is evident from the definition of $\Pi_\la$ that it may naturally be viewed as the set of integer points inside a certain convex polytope $P_\la\subset\mathbb R^{n^2}$. (This polytope can actually be obtained from a suitable type $C$ FFLV polytope via a diagonal linear transformation.) 

It is worth noting that our bijection $F$ is very much in the spirit of the bijection between the sets of integer points of a poset's order polytope and of its chain polytope (constructed in~\cite{stan}). Furthermore, it is even closer in spirit to the bijection between the sets of integer points of a marked order polytope and of a marked chain polytope (constructed in~\cite{ABS}). This is despite the fact that $P_\la$ is not a marked chain polytope per se.

On the other hand, the type $B$ Gelfand-Tsetlin polytope {\it is} a marked order polytope. This is observed in Section 4.3 of~\cite{ABS} and lets the authors suppose that a monomial basis in $L_\lambda$ is provided by (some modicfication of) the set $S(\lambda)$ obtained from $\Gamma_\lambda$ under the piecewise linear bijection with the corresponding marked chain polytope. We point out, however, that $\Pi_\lambda$ appears to be quite different from $S(\lambda)$ although, of course, both are in bijection with $\Gamma_\lambda$.
\end{remark}

\section{Ordered monomials}\label{ordmon}

Before carrying out our induction we introduce a few technical tools which, in particular, will let us eliminate arbitrary choices from the statement of Theorem~\ref{main}.

First we define a linear order on the set of positive $\mathfrak g$-roots, i.e. the set of integer pairs $1\le i<j \le\overline i$. We set $(i_1,j_1)\ll(i_2,j_2)$ whenever $i_1+j_1<i_2+j_2$ or $i_1+j_1=i_2+j_2$ and $i_1<i_2$ ordering the elements of a triangle from left to right and within a vertical column from bottom to top.

Next, we term a monomial $M\in\mathcal U(\mathfrak n^-)$ in the elements $f_{i,j}$ {\it ordered} if the elements occurring in $M$ are ordered according to $\ll$, i.e. if both $f_{i_1,j_1}$ and $f_{i_2,j_2}$ occur and $(i_1,j_1)\ll(i_2,j_2)$, then $f_{i_1,j_1}$ occurs to the left of $f_{i_2,j_2}$. Note that every ordered monomial is arranged and that the ordered monomials comprise a basis in $\mathcal U(\mathfrak n^-)$. For any monomial $X$ in the $f_{i,j}$ we denote $\ord(X)$ the ordered monomial with $\log\ord(X)=\log X$. For a number triangle $T=\{T_{i,j},1\le i<j\le\overline i\}$ with nonnegative integer elements we denote $\exp T$ the ordered monomial with $\log\exp(T)=T$.

Finally, $\ll$ induces a certain graded lexicographical order on the set of ordered monomials which we denote $\prec$. The grading of (any) monomial $M$ is given by $$\grad M=\sum_{1\le i<j\le \overline i}c_{i,j}(\log M)_{i,j},$$ where $c_{i,j}=1$ if $i+j<2n+1$ and $c_{i,j}=\frac12$ otherwise. For ordered monomials $M$ and $N$ we then write $M\prec N$ whenever $\grad M<\grad N$ or $\grad M=\grad N$ and the $\ll$-minimal pair $(i,j)$ with $(\log M)_{i,j}\neq(\log N)_{i,j}$ satisfies $(\log M)_{i,j}<(\log N)_{i,j}$.

An important property of $\prec$ is that it is {\it monomial} in the sense that if $M\prec N$ and $X\prec Y$, then $\ord(MX)\prec\ord(NY)$.

The key statement that we will be proving by induction can now be given as follows.
\begin{theorem}\label{ind}
Let $M$ be an ordered monomial with $\log M\notin\Pi_\la$. Then the vector $Mv_0\in L_\la$ can be expressed as a linear combination of vectors of the form $Kv_0$ with $K$ ordered and $K\prec M$.
\end{theorem}

It is evident that Theorem~\ref{ind} provides the special case of Theorem~\ref{main} in which all the monomials $M_T$ are ordered. However, it is actually not that hard to deduce the general case.
\begin{proposition}\label{implies}
Theorem~\ref{ind} implies Theorem~\ref{main}.
\end{proposition}

First we prove the following lemma.
\begin{lemma}\label{reorder}
Let $M$ and $N$ be two monomials with $\log M=\log N$. Then, as an element of $\mathcal U(\mathfrak n^-)$, the difference $M-N$ is a linear combination of ordered monomials $K$ with $\grad K\le \grad M=\grad N$. If $M$ and $N$ are arranged, then the last inequality is necessarily strict.
\end{lemma}
\begin{proof}
Let us proceed by induction on $\grad M$ and, within a specific $\grad M$, on the value $\sum_{i=1}^n (\log M)_{i,\overline i}$, i.e. the total number of elements of the form $f_{i,\overline i}$ in $M$. The base $M=1$ is trivial. 

Consider $M\neq 1$. Let us rearrange the monomial $M$ into $N$, i.e. let us obtain $N$ from $M$ by a series of operations of the form 
\begin{equation}\label{swap}
Xf_{i_1,j_1}f_{i_2,j_2}Y\rightarrow Xf_{i_2,j_2}f_{i_1,j_1}Y.
\end{equation} 
Note that
$$Xf_{i_1,j_1}f_{i_2,j_2}Y-Xf_{i_2,j_2}f_{i_1,j_1}Y=X[f_{i_1,j_1},f_{i_2,j_2}]Y$$
and that 
\begin{equation}\label{comgrad}
\grad(X[f_{i_1,j_1},f_{i_2,j_2}]Y)\le\grad(Xf_{i_1,j_1}f_{i_2,j_2}Y)=\grad(M).
\end{equation}
The inequality above is trivial when $i_1+j_1<2n+1$ or $i_2+j_2<2n+1$ and when $j_1=\overline{i_1}$ and $j_2=\overline{i_2}$ it follows from the commutation relation~(\ref{comrel}). In the latter case note that the monomial $X[f_{i_1,j_1},f_{i_2,j_2}]Y$ contains less elements of the form $f_{i,\overline i}$ than $M$ and $N$. We thus see that the difference $M-N$ is a linear combination of monomials considered previously in our induction. Applying the induction hypothesis to a monomial $Z$ in this linear combination, we replace $Z$ with the sum of $\ord(Z)$ and a linear combination of ordered monomials $K$ with $\grad K\le\grad Z\le\grad M$. This proves our claim.

Now, if $M$ and $N$ are arranged, we need only to perform operations of form~(\ref{swap}) for which $i_1+j_1<2n+1$ or $i_2+j_2<2n+1$. In this case the inequality in~(\ref{comgrad}) is necessarily strict.
\end{proof}

\begin{proof}[Proof of Proposition~\ref{implies}]
Theorem~\ref{ind} shows that the set of vectors $\{\exp(T)v_0,T\in\Pi_\lambda\}$ spans $L_\la$ and, with Corollary~\ref{gt} taken into account, we conclude that this set is, in fact, a basis. 

To prove the Proposition let us show that, given a set $D=\{M_T v_0,T\in\Pi_\la\}$ as in the statement of Theorem~\ref{main}, we can, employing Theorem~\ref{ind}, express a chosen $\exp(U) v_0$ with $U\in\Pi_\la$ as linear combination of the vectors in $D$. Indeed, let us proceed by induction on $\grad(\exp U)$ with the base $\exp U=1$ being trivial. 

For $\exp U\neq 1$ apply Lemma~\ref{reorder} to write $$\exp U=M_U+\sum_i c_i K_i,$$ for some numbers $c_i$ and ordered monomials $K_i$ with $\grad K_i<\grad(\exp U)$. Next, employing (iterating) Theorem~\ref{ind}, express any $K_i$ for which $\log K_i\notin\Pi_\lambda$ as a linear combination of ordered monomials $K$ with $\log K\in\Pi_\lambda$ and, moreover, $K\prec K_i$ whence $\grad(K)<\grad(\exp U)$. This last inequality permits us to invoke the induction hypothesis.
\end{proof}

With Proposition~\ref{implies} established the following two sections are devoted to an inductive proof of Theorem~\ref{ind}.

\section{Induction base}\label{basesec}

In this section we prove Theorem~\ref{ind} in the cases of $\la$ being a fundamental weight and $\la=2\omega_n$. To do so we make use of explicit descriptions of the corresponding modules. 

First, we describe the module $L_{\omega_1}$, which is the $(2n+1)$-dimensional vector representation of $\mathfrak{so}_{2n+1}$. We specify the actions of the root vectors $f_{i,j}$ in terms of a distinguished basis in $L_{\omega_1}$ consisting of the vectors $e_{-n},\ldots,e_n$. These actions are as follows.
\begin{enumerate}[label=(\roman*)]
\item The element $f_{i,j}$ with $1\le i<j\le n$ maps $e_i$ to $e_j$ and $e_{-j}$ to $-e_{-i}$, mapping all other $e_k$ to 0.
\item The element $f_{i,\overline j}$ with $1\le i<j\le n$ maps $e_i$ to $e_{-j}$ and $e_j$ to $-e_{-i}$, mapping all other $e_k$ to 0.
\item The element $f_{i,\overline i}$ with $1\le i\le n$ maps $e_i$ to $e_0$ and $e_0$ to $-2e_{-i}$, mapping all other $e_k$ to 0.
\end{enumerate}

Now it can be said that for $2\le i\le n-1$ the fundamental module $L_{\omega_i}$ is the exterior power $\wedge^i L_{\omega_1}$ with highest weight vector $e_1\wedge\ldots\wedge e_i$ and that $L_{2\omega_n}$ is the exterior power $\wedge^n L_{\omega_1}$ with highest weight vector $e_1\wedge\ldots\wedge e_n$. Such characterizations of the first $n-1$ fundamental modules can be found in~\cite{carter}. To verify the characterization of $L_{2\omega_n}$ one may, for example, compute its dimension via Weyl's dimension formula and then observe that $e_1\wedge\ldots\wedge e_n$ is indeed a highest weight vector of weight $2\omega_n$ in $\wedge^n L_{\omega_1}$.  

The remaining fundamental module $L_{\omega_n}$ is the so-called spin module which will be discussed separately towards the end of this section.

\begin{proposition}\label{base}
Suppose that $L_\lambda\cong \wedge^l L_{\omega_1}$ for some $l\in [1,n]$. Then Theorem~\ref{ind} holds.
\end{proposition}
\begin{proof}
We are given an ordered monomial $M$ with $\log M\notin\Pi_\lambda$ and we are to show that $Mv_0\in\Omega$, where $\Omega\subset\mathcal U(\mathfrak n^-)$ is the subspace spanned by vectors of the form $Kv_0$ with $K\prec M$. We assume $v_0=e_1\wedge\ldots\wedge e_l$ and denote $T=\log M$.

Our weight $\la$ is given by $\la_1=\ldots=\la_l=1$ and $\la_i=0$ for $i>l$. Hence, for Dyck paths $d$ starting in one of $(1,2),\ldots,(l,l+1)$ and ending either in one of $(l+1,l+2),\ldots,(n-1,n)$ or anywhere in the rightmost vertical column we have $M(\la,d)=1$. For all other Dyck paths $d$ we have  $M(\la,d)=0$. Therefore, by the definition of $\Pi_\lambda$, the fact that $T\notin\Pi_\lambda$ leaves us with four possibilities.
\begin{enumerate}[label=(\Roman*)]
\item We have $T_{i,j}>0$ for some $i>l$ or some $j<l+1$.
\item We have $T_{i_1,j_1}>0$ and $T_{i_2,j_2}>0$ for two distinct pairs $(i_1,j_1)$ and $(i_2,j_2)$ with $i_1\le i_2\le l$ and $l+1\le j_1\le j_2$. The meaning of these inequalities is that there exists a Dyck path passing first through $(i_1,j_1)$ and then through $(i_2,j_2)$.
\item We have $T_{i,j}>1$ for some $i\le l$ and $j\ge l+1$ with $i+j<2n+1$.
\item We have $T_{i,j}>2$ for some $i\le l$ and $i+j=2n+1$. 
\end{enumerate}

To visualize each of these possibilities and (especially) the cases they will be broken up into it is helpful to consider a partitioning of the set of pairs $1\le i<j\le\overline i$ into six different subsets. We provide the following diagram demonstrating this partitioning in the case of $n=5$ and $l=3$.

\vspace{5mm}
\setlength{\unitlength}{1mm}
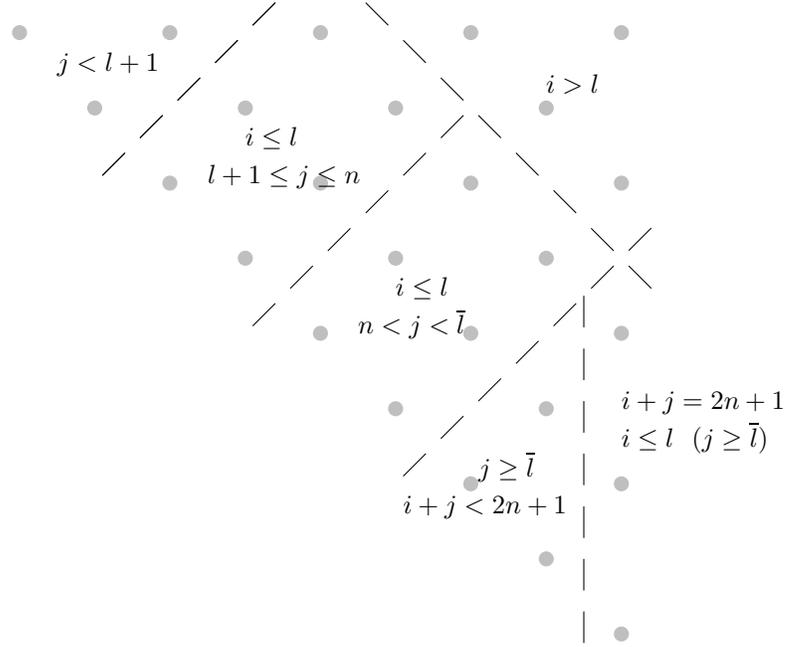
\begin{center}
\begin{picture}(130,90)
\color{lightgray}
\put(20,90){\circle*{2}}
\put(40,90){\circle*{2}}
\put(60,90){\circle*{2}}
\put(80,90){\circle*{2}}
\put(100,90){\circle*{2}}
\put(30,80){\circle*{2}}
\put(50,80){\circle*{2}}
\put(70,80){\circle*{2}}
\put(90,80){\circle*{2}}
\put(40,70){\circle*{2}}
\put(60,70){\circle*{2}}
\put(80,70){\circle*{2}}
\put(100,70){\circle*{2}}
\put(50,60){\circle*{2}}
\put(70,60){\circle*{2}}
\put(90,60){\circle*{2}}
\put(60,50){\circle*{2}}
\put(80,50){\circle*{2}}
\put(100,50){\circle*{2}}
\put(70,40){\circle*{2}}
\put(90,40){\circle*{2}}
\put(80,30){\circle*{2}}
\put(100,30){\circle*{2}}
\put(90,20){\circle*{2}}
\put(100,10){\circle*{2}}

\color{black}
\multiput(66,94)(5,-5){8}{\line(1,-1){3}}
\multiput(54,94)(-5,-5){5}{\line(-1,-1){3}}
\multiput(54,94)(-5,-5){5}{\line(-1,-1){3}}
\multiput(95,55)(0,-7){7}{\line(0,-1){4.2}}
\multiput(71,31)(5,5){7}{\line(1,1){3}}
\multiput(79,79)(-5,-5){6}{\line(-1,-1){3}}

\put(25,85){$j<l+1$}

\put(50,75){$i\le l$}
\put(45,70){$l+1\le j\le n$}

\put(70,55){$i\le l$}
\put(65,50){$n<j<\overline l$}

\put(90,82){$i>l$}


\put(100,40){$i+j=2n+1$}
\put(100,35){$i\le l$ {} ($j\ge\overline l$)}

\put(81,31){$j\ge\overline l$}
\put(71,26){$i+j<2n+1$}

\end{picture}
\end{center}

\vspace{-15mm}
Our four possibilities will further be split up into cases (especially possibility (II)) and altogether there is quite a number of different situations to discuss. However, the outline of the argument will always be the same in spirit or, rather, conform to one of the following scenarios. 

In a few trivial situations we will simply have $Mv_0=0$. In all other situations we will present an arranged monomial $M'$ with $\log M'=\log M$ and show that $M'v_0\in\Omega$ which will suffice by Lemma~\ref{reorder}. Again, if we are lucky, we will simply have $M'v_0=0$, however, in general, this is not the case. In general, we will use the explicit actions of the $f_{i,j}$ provided by (i), (ii) and (iii) to express $M'v_0$ as a certain linear combination of vectors of the form $Nv_0$, where the monomials $N$ are of the following type. \label{outline}Each of these $N$ is obtained from $M'$ by replacing the product of the two ``problematic'' elements ($f_{i_1,j_1}f_{i_2,j_2}$ in possibility (II) and $f_{i,j}^2$ in possibility (III)) with a certain different expression, less (with respect to $\prec$) than the product being replaced. This then means that we have $\ord(N)\prec M$ and we are left to show that $N-\ord(N)\in\Omega$. This last assertion is proved with the help of Lemma~\ref{reorder} and, at times, a couple additional remarks.

We first deal with possibility (I). If we have $T_{i,2n+1-i}>0$ for some $i>l$, then we may assume that $i$ is the largest possible, i.e. the rightmost multiple in the monomial $M$ is $f_{i,\overline i}$. However, from (iii) we see that $f_{i,\overline i}$ maps each of $e_1,\ldots,e_l$ to 0, hence it also maps $v_0$ to 0 and the assertion is trivial.

Otherwise, we have $T_{i,j}>0$ with either $i>l$ or $j<l+1$ and with $i+j<2n+1$. Consider the monomial $M'$ obtained from $M$ by shifting one $f_{i,j}$ to the very right and preserving the order of the other elements. By Lemma~\ref{reorder} the difference $M-M'$ is a linear combination of monomials $K\prec M$. However, we have $M'v_0=0$ due to $f_{i,j}v_0=0$ which follows from (i) and (ii).

We move on to possibility (II) which will be split up into numerous cases corresponding to the numerous ways $(i_1,j_1)$ and $(i_2,j_2)$ can be positioned.

{\it Case 1.} We have $j_1\le j_2\le n$. Consider the monomial $M'$ obtained from $M$ by shifting $f_{i_1,j_1}$ and $f_{i_2,j_2}$ to the very right and preserving the order of the remaining elements, i.e. $M'=X f_{i_1,j_1}f_{i_2,j_2}$. By Lemma~\ref{reorder} it suffices to show that $M'v_0\in\Omega$. 

If $i_1=i_2$ or $j_1=j_2$ from (i) we immediately have $M'v_0=f_{i_1,j_1}f_{i_2,j_2}v_0=0$. Otherwise, we claim that $f_{i_1,j_1}f_{i_2,j_2} v_0=-f_{i_1,j_2}f_{i_2,j_1}v_0$. Indeed, let us express $v_0$ as $e_{i_1}\wedge e_{i_2}\wedge E$. Clearly, $$f_{i_1,j_1}f_{i_2,j_2} E=f_{i_1,j_2}f_{i_2,j_1}E=0$$ and we are to show that $$f_{i_1,j_1}f_{i_2,j_2}(e_{i_1}\wedge e_{i_2})=-f_{i_1,j_2}f_{i_2,j_1}(e_{i_1}\wedge e_{i_2}).$$ However, with the help of (i) it is easily seen that both of the above expressions equal $e_{j_1}\wedge e_{j_2}$.

We have obtained $M'v_0=-Xf_{i_1,j_2}f_{i_2,j_1}v_0$. By Lemma~\ref{reorder} we can rewrite the right-hand side as the sum of $-\ord(Xf_{i_1,j_2}f_{i_2,j_1}v_0)$ and an element of $\Omega$. However, since $i_1+j_2>i_1+j_1$ and $i_2+j_1>i_1+j_1$, we also have $\ord(Xf_{i_1,j_2}f_{i_2,j_1}v_0)\prec M$. 

{\it Case 2.} We have $j_1\le n$ while $n<j_2<\overline l$. The proof repeats the one from Case 1 verbatim, except that the equality $j_1=j_2$ is now impossible and that $f_{i_1,j_1}f_{i_2,j_2}(e_{i_1}\wedge e_{i_2})$ and $-f_{i_1,j_2}f_{i_2,j_1}(e_{i_1}\wedge e_{i_2})$ now both equal $e_{j_1}\wedge e_{-\overline{j_2}}$.

{\it Case 3.} We have $j_1\le n$ while $j_2\ge\overline l$ and $i_2+j_2<2n+1$. We will also assume that $T_{i,\overline i}=0$ for all $i\ge i_1$, since otherwise we would be within Case 7. 

Similarly to the previous two cases we consider $M'=X f_{i_1,j_1}f_{i_2,j_2}$ obtained from $M$ by shifting $f_{i_1,j_1}f_{i_2,j_2}$ to the very right and show that $M'v_0\in\Omega$.

First we assume that $i_1<i_2$. Invoking (i), (ii) and (iii), write down the following equalities (we omit the wedge product for brevity)
$$f_{i_1,j_1}f_{i_2,j_2}(e_{i_1}e_{i_2}e_{\overline{j_2}})=e_{j_1}e_{\overline{j_2}}e_{-\overline{j_2}}-e_{j_1}e_{i_2}e_{-i_2},$$ 
$$f_{i_1,j_2}f_{i_2,j_1}(e_{i_1}e_{i_2}e_{\overline{j_2}})=e_{-\overline{j_2}}e_{j_1}e_{\overline{j_2}}-e_{i_1}e_{j_1}e_{-i_1},$$ 
$$f_{i_1,\overline{i_2}}f_{\overline{j_2},j_1}(e_{i_1}e_{i_2}e_{\overline{j_2}})=e_{-i_2}e_{i_2}e_{j_1}-e_{i_1}e_{-i_1}e_{j_1}$$ and 
$$f_{\overline{j_2},j_1}f_{i_1,\overline{i_1}}f_{i_2,\overline{i_2}}(e_{i_1}e_{i_2}e_{\overline{j_2}})=-2e_{i_1}e_{-i_1}e_{j_1}.$$ 
We obtain 
\begin{equation}\label{case3}
f_{i_1,j_1}f_{i_2,j_2}v_0=(f_{i_1,\overline{i_2}}f_{\overline{j_2},j_1}-f_{i_1,j_2}f_{i_2,j_1}-f_{\overline{j_2},j_1}f_{i_1,\overline{i_1}}f_{i_2,\overline{i_2}})v_0
\end{equation}
(we express $v_0=e_{i_1}\wedge e_{i_2}\wedge e_{\overline{j_2}}\wedge E$ and note that all of the above monomials map $E$ to 0).

Now, similarly to Case 1 we deduce that $M'v_0$ is congruent to $$(\ord(Xf_{i_1,\overline{i_2}}f_{\overline{j_2},j_1})-\ord(Xf_{i_1,j_2}f_{i_2,j_1})-\ord(Xf_{\overline{j_2},j_1}f_{i_1,\overline{i_1}}f_{i_2,\overline{i_2}}))v_0$$ modulo $\Omega$ and observe that each of the monomials in the expression above lies in $\Omega$. To prove this last assertion we verify that all the $f_{i,j}$ appearing in the right-hand side of~(\ref{case3}) satisfy $i+j>i_1+j_1$: 
\begin{equation}\label{i1oi2}
i_1+\overline{i_2}> i_1+j_2\ge i_1+j_1,
\end{equation}
\begin{equation}\label{oj2j1}
\overline{j_2}+j_1>i_2+j_1\ge i_1+j_1,
\end{equation}
\begin{equation}\label{i1andi2}
i_1+\overline{i_1}=i_2+\overline{i_2}=2n+1>i_2+j_2>i_1+j_1
\end{equation}
and $(i_1,j_2)$ and $(i_2,j_1)$ are considered trivially (as in Case~1).
We also make use of the fact that $Xf_{\overline{j_2},j_1}f_{i_1,\overline{i_1}}f_{i_2,\overline{i_2}}$ is arranged which is due to the assumption made at the beginning of this case.

Finally, if we have $i_1=i_2$, we have the easily obtainable identity $$f_{i_1,j_1}f_{i_2,j_2}v_0=-\frac12f_{\overline{j_2},j_1}f_{i_1,\overline{i_1}}^2v_0.$$ The rest of the argument is analogous and makes use of~(\ref{oj2j1}) and~(\ref{i1andi2}). 

{\it Case 4.} We have $n<j_1\le j_2<\overline l$. The proof repeats the one from Case 1 verbatim, except that $f_{i_1,j_1}f_{i_2,j_2}(e_{i_1}\wedge e_{i_2})$ and $-f_{i_1,j_2}f_{i_2,j_1}(e_{i_1}\wedge e_{i_2})$ now both equal $e_{-\overline{j_1}}\wedge e_{-\overline{j_2}}$.

{\it Case 5.} We have $n<j_1<\overline l\le j_2$ and $i_2+j_2<2n+1$. The proof repeats the one from Case 3 verbatim except for a substitution of $e_{-\overline{j_1}}$ for $e_{j_1}$ and the remark that having $T_{i,\overline i}>0$ for some $i\ge i_1$ would let us reduce to Case 8 (and not Case 7).

{\it Case 6.} We have $\overline l\le j_1\le j_2$ and $i_2+j_2<2n+1$. We will also assume that $T_{i,\overline i}=0$ for all $i_1\le i\le \overline{j_1}$, since otherwise we would be within Case 9. 

Here to define $M'$ we don't shift $f_{i_1,j_1}f_{i_2,j_2}$ to the very right but instead we shift it to the right of all elements except those of form $f_{i,\overline i}$ with $i>\overline{j_1}$. We denote $M'=Xf_{i_1,j_1}f_{i_2,j_2}Y$. The vector $Yv_0$ is a linear combination of vectors of two forms: either $e_{i_1}e_{i_2}e_{\overline{j_2}}e_{\overline{j_1}}E$ or $e_0e_{i_1}e_{i_2}e_{\overline{j_2}}e_{\overline{j_1}}E$ with $f_{i_1,j_1}f_{i_2,j_2}E=0$ in both cases.

Next, in the spirit of the previous cases, we assume that $i_1<i_2$ and $j_1<j_2$ and write down the equalities
\begin{multline*}
f_{i_1,j_1}f_{i_2,j_2}(e_{i_1}e_{i_2}e_{\overline{j_2}}e_{\overline{j_1}})=e_{-\overline{j_1}}e_{-\overline{j_2}}e_{\overline{j_2}}e_{\overline{j_1}}-e_{-\overline{j_1}}e_{i_2}e_{-i_2}e_{\overline{j_1}}-\\e_{i_1}e_{-\overline{j_2}}e_{\overline{j_2}}e_{-i_1}+e_{i_1}e_{i_2}e_{-i_1}e_{-i_2},
\end{multline*}
\begin{multline*}
f_{i_1,j_2}f_{i_2,j_1}(e_{i_1}e_{i_2}e_{\overline{j_2}}e_{\overline{j_1}})=e_{-\overline{j_2}}e_{-\overline{j_1}}e_{\overline{j_2}}e_{\overline{j_1}}-e_{-\overline{j_2}}e_{i_2}e_{\overline{j_2}}e_{-i_2}-\\e_{i_1}e_{-\overline{j_1}}e_{-i_1}e_{\overline{j_1}}+e_{i_1}e_{i_2}e_{-i_2}e_{-i_1},
\end{multline*}
\begin{multline*}
f_{i_1,\overline{i_2}}f_{\overline{j_2},j_1}(e_{i_1}e_{i_2}e_{\overline{j_2}}e_{\overline{j_1}})=e_{-i_2}e_{i_2}e_{-\overline{j_1}}e_{\overline{j_1}}-e_{-i_2}e_{i_2}e_{\overline{j_2}}e_{-\overline{j_2}}-\\e_{i_1}e_{-i_1}e_{-\overline{j_1}}e_{\overline{j_1}}+e_{i_1}e_{-i_1}e_{\overline{j_2}}e_{-\overline{j_2}},
\end{multline*}
$$f_{\overline{j_2},j_1}f_{i_1,\overline{i_1}}f_{i_2,\overline{i_2}}(e_{i_1}e_{i_2}e_{\overline{j_2}}e_{\overline{j_1}})=-2e_{i_1}e_{-i_1}e_{-\overline{j_1}}e_{\overline{j_1}}+2e_{i_1}e_{-i_1}e_{\overline{j_2}}e_{-\overline{j_2}},$$ 
$$f_{i_1,\overline{i_2}}f_{\overline{j_2}}f_{\overline{j_1}}(e_{i_1}e_{i_2}e_{\overline{j_2}}e_{\overline{j_1}})=-2e_{-i_2}e_{i_2}e_{\overline{j_2}}e_{-\overline{j_2}}+2e_{i_1}e_{-i_1}e_{\overline{j_2}}e_{-\overline{j_2}},$$ 
$$f_{i_1,\overline{i_1}}f_{i_2,\overline{i_2}}f_{\overline{j_2}}f_{\overline{j_1}}(e_{i_1}e_{i_2}e_{\overline{j_2}}e_{\overline{j_1}})=4e_{i_1}e_{-i_1}e_{\overline{j_2}}e_{-\overline{j_2}}.$$
We also have
$$f_{i_1,j_1}f_{i_2,j_2}(e_0e_{i_1}e_{i_2}e_{\overline{j_2}}e_{\overline{j_1}})=e_0f_{i_1,j_1}f_{i_2,j_2}(e_{i_1}e_{i_2}e_{\overline{j_2}}e_{\overline{j_1}}),$$
$$f_{i_1,j_2}f_{i_2,j_1}(e_0e_{i_1}e_{i_2}e_{\overline{j_2}}e_{\overline{j_1}})=e_0f_{i_1,j_2}f_{i_2,j_1}(e_{i_1}e_{i_2}e_{\overline{j_2}}e_{\overline{j_1}}),$$
$$f_{i_1,\overline{i_2}}f_{\overline{j_2},j_1}(e_0e_{i_1}e_{i_2}e_{\overline{j_2}}e_{\overline{j_1}})=e_0f_{i_1,\overline{i_2}}f_{\overline{j_2},j_1}(e_{i_1}e_{i_2}e_{\overline{j_2}}e_{\overline{j_1}}),$$
$$f_{\overline{j_2},j_1}f_{i_1,\overline{i_1}}f_{i_2,\overline{i_2}}(e_0e_{i_1}e_{i_2}e_{\overline{j_2}}e_{\overline{j_1}})=-2e_{-i_2}e_{0}e_{i_2}e_{-\overline{j_1}}e_{\overline{j_1}}+2e_{-i_2}e_{0}e_{i_2}e_{\overline{j_2}}e_{-\overline{j_2}},$$ 
$$f_{i_1,\overline{i_2}}f_{\overline{j_2}}f_{\overline{j_1}}(e_0e_{i_1}e_{i_2}e_{\overline{j_2}}e_{\overline{j_1}})=-2e_{-\overline{j_1}}e_{-i_2}e_{i_2}e_{0}e_{\overline{j_1}}+2e_{-\overline{j_1}}e_{i_1}e_{-i_1}e_{0}e_{\overline{j_1}},$$ 
$$f_{i_1,\overline{i_1}}f_{i_2,\overline{i_2}}f_{\overline{j_2}}f_{\overline{j_1}}(e_0e_{i_1}e_{i_2}e_{\overline{j_2}}e_{\overline{j_1}})=4e_{-\overline{j_1}}e_{0}e_{i_2}e_{-i_2}e_{\overline{j_1}}.$$
We see that we have the same linear relation for the six right-hand sides in both cases and we end up with $$M'v_0=X(f_{i_1,\overline{i_2}}f_{\overline{j_2},j_1}-f_{i_1,j_2}f_{i_2,j_1}-f_{\overline{j_2},j_1}f_{i_1,\overline{i_1}}f_{i_2,\overline{i_2}}-f_{i_1,\overline{i_2}}f_{\overline{j_2}}f_{\overline{j_1}}+f_{i_1,\overline{i_1}}f_{i_2,\overline{i_2}}f_{\overline{j_2}}f_{\overline{j_1}})Yv_0.$$

We complete the argument as in the previous cases, noting that the last three monomials in the (expanded) right-hand side above are arranged due to the assumption we made at the beginning of this case. Here we make use of~(\ref{i1oi2}),~\ref{oj2j1},~\ref{i1andi2} and of 
\begin{equation}\label{oj1andoj2}
\overline{j_1}+j_1=\overline{j_2}+j_2=2n+1>i_2+j_2>i_1+j_1.
\end{equation}

To complete the consideration of this case we are left to deal with the situations in which $i_1=i_2$ or $j_1=j_2$. When $i_1=i_2$ we write 
$$f_{i_1,j_1}f_{i_2,j_2}(e_{i_1}e_{\overline{j_2}}e_{\overline{j_1}})=-e_{-\overline{j_2}}e_{\overline{j_2}}e_{-i_1}-e_{-\overline{j_1}}e_{-i_2}e_{\overline{j_1}},$$
$$f_{\overline{j_2},j_1}f_{i_1,\overline{i_1}}^2(e_{i_1}e_{\overline{j_2}}e_{\overline{j_1}})=-2e_{-i_1}e_{-\overline{j_1}}e_{\overline{j_1}}+2e_{-i_1}e_{\overline{j_2}}e_{-\overline{j_2}},$$
$$f_{i_1,\overline{i_1}}^2f_{\overline{j_2}}f_{\overline{j_1}}(e_{i_1}e_{\overline{j_2}}e_{\overline{j_1}})=4e_{-i_1}e_{\overline{j_2}}e_{-\overline{j_2}}$$
and
$$f_{i_1,j_1}f_{i_2,j_2}(e_0e_{i_1}e_{\overline{j_2}}e_{\overline{j_1}})=e_0f_{i_1,j_1}f_{i_2,j_2}(e_{i_1}e_{\overline{j_2}}e_{\overline{j_1}}),$$
$$f_{\overline{j_2},j_1}f_{i_1,\overline{i_1}}^2(e_0e_{i_1}e_{\overline{j_2}}e_{\overline{j_1}})=-2e_{-i_1}e_0e_{-\overline{j_1}}e_{\overline{j_1}}+2e_{-i_1}e_0e_{\overline{j_2}}e_{-\overline{j_2}},$$
$$f_{i_1,\overline{i_1}}^2f_{\overline{j_2}}f_{\overline{j_1}}(e_0e_{i_1}e_{\overline{j_2}}e_{\overline{j_1}})=4e_{-\overline{j_1}}e_0e_{-i_1}e_{\overline{j_1}}=4e_{-\overline{j_1}}e_0e_{-i_2}e_{\overline{j_1}}$$
to conclude $$M'v_0=\frac12X(-f_{\overline{j_2},j_1}f_{i_1,\overline{i_1}}^2+f_{i_1,\overline{i_1}}^2f_{\overline{j_2}}f_{\overline{j_1}})Yv_0.$$

When $j_1=j_2$ we similarly derive $$M'v_0=\frac12X(-f_{i_1,\overline{i_2}}f_{\overline{j_1}}^2+f_{i_1,\overline{i_1}}f_{i_2,\overline{i_2}}f_{\overline{j_1}}^2)Yv_0.$$ In either situation the argument is then finished off as when we had $i_1<i_2$ and $j_1<j_2$.

{\it Case 7.} We have $j_1\le n$ and $i_2+j_2=2n+1$. We assume that $i_2$ is the largest $i$ for which $T_{i,2n+1-i}>0$, i.e. the rightmost element in $M$ is $f_{i_2,j_2}=f_{i_2,\overline{i_2}}$.

We define $M'$ by shifting $f_{i_1,j_1}$ to the right of all elements other than the last $f_{i_2,\overline{i_2}}$. The rest of the argument is analogous to Cases 1 - 5 and makes use of 
$$f_{i_1,j_1}f_{i_2,\overline{i_2}}(e_{i_1}e_{i_2})=e_{j_1}e_0,$$
$$f_{i_2,j_1}f_{i_1,\overline{i_1}}(e_{i_1}e_{i_2})=e_0e_{j_1},$$
$$f_{i_1,j_1}f_{i_2,\overline{i_2}}v_0=-f_{i_2,j_1}f_{i_1,\overline{i_1}}v_0$$
when $i_1<i_2$ and of $f_{i_1,j_1}f_{i_2,\overline{i_2}}v_0=0$ when $i_1=i_2$. (In the former case we make use of~(\ref{i1andi2}) substituting the first ``$>$'' for a ``=''.)

{\it Case 8.} We have $n<j_1<\overline l$ and $i_2+j_2=2n+1$. The argument repeats the one in Case 7 verbatim except for a substitution of $e_{-\overline{j_1}}$ for $e_{j_1}$.

{\it Case 9.} We have $j_1\ge\overline l$ and $i_2+j_2=2n+1$. We assume that $i_2$ is the largest $i$ with $T_{i,2n+1-i}>0$ and $i_1\le i\le \overline{j_1}$.

We define $M'$ by shifting $f_{i_1,j_1}$ to the immediate left of the rightmost $f_{i_2,\overline{i_2}}$ to obtain $M'=Xf_{i_1,j_1}f_{i_2,\overline{i_2}}Y$ with $Y$ containing only elements of the form $f_{i,\overline i}$ with $i>i_2$. We see that $Yv_0$ is a linear combination of vectors of the forms $e_{i_1}e_{i_2}e_{\overline{j_1}}E$ and $e_0e_{i_1}e_{i_2}e_{\overline{j_1}}E$ with $f_{i_1,j_1}f_{i_2,\overline{i_2}}E=0$.

First we assume that $i_1<i_2<\overline{j_1}$ and write 
$$f_{i_1,j_1}f_{i_2,\overline{i_2}}(e_{i_1}e_{i_2}e_{\overline{j_1}})=e_{-\overline{j_1}}e_0e_{\overline{j_1}}-e_{i_1}e_0e_{-i_1},$$
$$f_{i_2,j_1}f_{i_1,\overline{i_1}}(e_{i_1}e_{i_2}e_{\overline{j_1}})=e_0e_{-\overline{j_1}}e_{\overline{j_1}}-e_0e_{i_2}e_{-i_2},$$
$$f_{i_1,\overline{i_2}}f_{\overline{j_1}}(e_{i_1}e_{i_2}e_{\overline{j_1}})=e_{-i_2}e_{i_2}e_0-e_{i_1}e_{-i_1}e_0,$$
$$f_{i_1,\overline{i_1}}f_{i_2,\overline{i_2}}f_{\overline{j_1}}(e_{i_1}e_{i_2}e_{\overline{j_1}})=-2e_0e_{i_2}e_{-i_2}$$
and
$$f_{i_1,j_1}f_{i_2,\overline{i_2}}(e_0e_{i_1}e_{i_2}e_{\overline{j_1}})=-2e_{-i_2}e_{-\overline{j_1}}e_{i_2}e_{\overline{j_1}}+2e_{-i_2}e_{i_1}e_{i_2}e_{-i_1},$$
$$f_{i_2,j_1}f_{i_1,\overline{i_1}}(e_0e_{i_1}e_{i_2}e_{\overline{j_1}})=-2e_{-i_1}e_{i_1}e_{-\overline{j_1}}e_{\overline{j_1}}+2e_{-i_1}e_{i_1}e_{i_2}e_{-i_2},$$
$$f_{i_1,\overline{i_2}}f_{\overline{j_1}}(e_0e_{i_1}e_{i_2}e_{\overline{j_1}})=-2e_{-\overline{j_1}}e_{-i_2}e_{i_2}e_{\overline{j_1}}+2e_{-\overline{j_1}}e_{i_1}e_{-i_1}e_{\overline{j_1}},$$
$$f_{i_1,\overline{i_1}}f_{i_2,\overline{i_2}}f_{\overline{j_1}}(e_0e_{i_1}e_{i_2}e_{\overline{j_1}})=4e_{-\overline{j_1}}e_{i_1}e_{-i_1}e_{\overline{j_1}}$$
obtaining
$$M'v_0=X(-f_{i_2,j_1}f_{i_1,\overline{i_1}}-f_{i_1,\overline{i_2}}f_{\overline{j_1}}+f_{i_1,\overline{i_1}}f_{i_2,\overline{i_2}}f_{\overline{j_1}})Yv_0.$$

A particularity of this case is that the monomials $Xf_{i_2,j_1}f_{i_1,\overline{i_1}}Y$ and $Xf_{i_1,\overline{i_1}}f_{i_2,\overline{i_2}}f_{\overline{j_1}}Y$ will not be arranged if $X$ contains any $f_{i,\overline i}$ with $i>i_1$. However, we still claim that the expressions $$Xf_{i_2,j_1}f_{i_1,\overline{i_1}}Y-\ord(Xf_{i_2,j_1}f_{i_1,\overline{i_1}}Y)$$ and $$Xf_{i_1,\overline{i_1}}f_{i_2,\overline{i_2}}f_{\overline{j_1}}Y-\ord(Xf_{i_1,\overline{i_1}}f_{i_2,\overline{i_2}}f_{\overline{j_1}}Y)$$ lie in $\Omega$. 

Indeed, let us consider $Xf_{i_2,j_1}f_{i_1,\overline{i_1}}Y$ with $X$ containing a $f_{i,\overline i}$ with $i>i_1$. Let $Z$ be the arranged monomial obtained from $Xf_{i_2,j_1}f_{i_1,\overline{i_1}}Y$ by shifting the $f_{i_1,\overline{i_1}}$ to the immediate left of the leftmost $f_{i,\overline i}$ with $i>i_1$. This shift consists of a series of operations of the form $X'f_{i,\overline i}f_{i_1,\overline{i_1}}Y'\rightarrow X'f_{i_1,\overline{i_1}}f_{i,\overline i}Y'$ with $i>i_1$.
Note that
$$X'f_{i,\overline i}f_{i_1,\overline{i_1}}Y'-X'f_{i_1,\overline{i_1}}f_{i,\overline i}Y'=-2X'f_{i_1,\overline{i}}Y'$$
by~(\ref{comrel}). However, by definition of $M'$ we must have $i<\overline{j_1}$ whence $i_1+\overline{i}>i_1+j_1$ and $X'f_{i_1,\overline{i}}Y'\in\Omega$. Therefore, we may conclude that $Xf_{i_2,j_1}f_{i_1,\overline{i_1}}Y-Z\in\Omega$. We also have $Z-\ord(Z)\in\Omega$ by Lemma~\ref{reorder} and our assertion follows. The monomial $Xf_{i_1,\overline{i_1}}f_{i_2,\overline{i_2}}f_{\overline{j_1}}Y$ is considered analogously and the argument is completed as in the previous cases. (Note that in~(\ref{i1oi2}) the ``$>$'' is replaced with a ``$=$'' but the second inequality is strict since $j_2=\overline{i_2}>j_1$ due to our current assumption. Also note that in~(\ref{oj1andoj2}) the first ``$>$'' is replaced with a ``$=$''.)

When $i_1=i_2$ we have $$M'v_0=\frac12Xf_{i_1,\overline{i_1}}^2f_{\overline{j_1}}Yv_0.$$ When $i_2=\overline{j_1}$ we have $$M'v_0=\frac12Xf_{i_1,\overline{i_1}}f_{\overline{j_1}}^2Yv_0.$$ In both cases the argument is completed just like when we had $i_1<i_2<\overline{j_1}$, however, when $i_2=\overline{j_1}$ the monomial $M'$ is to be defined by shifting $f_{i_1,j_1}$ to the immediate left of the {\it leftmost} $f_{i_2,\overline{i_2}}=f_{\overline{j_1}}$. This last adjustment is necessary in order to avoid having a $f_{\overline{j_1}}$ in $X$. Otherwise, when shifting our $f_{i_1,\overline{i_1}}$ in $Xf_{i_1,\overline{i_1}}f_{\overline{j_1}}^2Y$ to the left we would be commuting it with $f_{\overline{j_1}}$ and obtaining $f_{i_1,j_1}$.

We have completed the consideration of possibility (II) and are left to deal with possibilities (III) and (IV).

If we are within possibility (III) and we have $j<\overline l$, then we argue as in possibility (I). We define $M'$ by shifting the $f_{i,j}^2$ to very left. It is then easily seen that $M'v_0=0$ and that $M-M'\in\Omega$ by Lemma~\ref{reorder}.

If we, however, are within possibility (III) and have $j\ge\overline l$, then our argument may be viewed as an adaptation of the argument in Case 6 above to the situation $i_1=i_2$ and $j_1=j_2$.  Namely, we define $M'$, $X$ and $Y$ as in Case 6 (setting $i_1=i_2=i$ and $j_1=j_2=j$) and obtain the relation $$M'v_0=-\frac12Xf_{i,\overline i}^2f_{\overline j,j}^2Yv_0.$$ The argument is then completed analogously to Case 6.

Finally, if we are within possibility (IV), we simply make use of the fact that $f_{i,\overline i}^3$ annihilates $L_\lambda$.
\end{proof}

We now discuss the remaining fundamental weight $\omega_n$.

\begin{proposition}
Theorem~\ref{ind} holds if $\la=\omega_n$.
\end{proposition}
\begin{proof}
An explicit description of the spin representation $L_{\omega_n}$ may be found in~\cite{carter}. We will utilize the following properties. The spin representation has dimension $2^n$ and a basis comprised of vectors $e_I$ with $I\in\{0,1\}^n$. Vector $e_I$ has weight $\frac12\sum_{j=1}^n (-1)^{I_j}\beta_j$, in particular, $v_0=e_{\{0,\ldots,0\}}$ is the highest vector. We will also need the fact that if $I_j=0$, then $f_{j,\overline j}e_{I}$ is a nonzero multiple of $e_{I'}$, where $I'$ is obtained from $I$ by setting $I'_j=1$.

Observe that $\Pi_{\omega_n}$ is comprised of all $U$ with $U_{i,\overline i}\in\{0,1\}$ for all $1\le i\le n$ and $U_{i,j}=0$ for all $i+j<2n+1$. In view of the above properties, this already shows that the vectors $\exp(U)v_0,U\in\Pi_{\omega_n}$ comprise a basis in $L_{\omega_n}$.

Now, let $M$ be an ordered monomial with $T=\log M\notin\Pi_{\omega_n}$. If $T_{i,j}=0$ for all $i+j<2n+1$, then we must have $T_{i,\overline i}>1$ for some $i$ to avoid having $T\in\Pi_{\omega_n}$. However, we then would evidently have $Mv_0=0$ for weight reasons. If, on the contrary, we have $T_{i,j}>0$ for some $i+j<2n+1$, then we may decompose $Mv_0$ in the basis $\{\exp(U)v_0,U\in\Pi_{\omega_n}\}$ and observe that we have $U\prec T$ for any $U\in\Pi_{\omega_n}$. 
\end{proof}

\section{Induction step}

In this section we complete the proof of Theorem~\ref{ind} by transitioning from the cases discussed in the previous section to the general case. Fortunately, this transition is nowhere as tedious as the above case by case proof. It is enabled by the following Minkowski type property.

\begin{lemma}\label{mink}
Suppose that $\la\neq 0$ and is neither a fundamental weight nor $2\omega_n$. Let $l$ be the minimal $i$ such that $a_i>0$ and let $T\in\Pi_\la$. If $l<n$, set $\varepsilon=\omega_l$ and else set $\varepsilon=2\omega_n$ and denote $\lambda'=\lambda-\varepsilon$. Then there exists $T'\in\Pi_{\lambda'}$ such that $T-T'\in\Pi_\varepsilon$.
\end{lemma}
\begin{proof}
We define a weakening $\lll$ of the order $\ll$ on the set of pairs $(i,j)$ with $1\le i<j\le\overline i$. For $(i_1,i_2)\neq(j_1,j_2)$ we write $(i_1,j_1)\lll(i_2,j_2)$ whenever $i_1\le i_2$ and $j_1\le j_2$. In fact, one sees that having $(i_1,j_1)\lll(i_2,j_2)$ is equivalent to having a Dyck path passing first through $(i_1,j_1)$ and then through $(i_2,j_2)$.

Let $\mathcal T$ be the set of pairs $(i,j)$ such that $T_{i,j}>0$ and let $\mathcal M$ be the set of pairs $(i,j)$ that are $\lll$-minimal elements of $\mathcal T$ and have the property $i\le l$ . Note that any $(i,j)\in\mathcal T$ satisfies $j\ge l+1$, since otherwise we wouldn't have $T\in\Pi_\lambda$. That is because we would then have a Dyck path $d$ passing through $(i,j)$ and starting and ending left of $T_{l,l+1}$ yielding $S(T,d)>M(\la,d)=0$. Also note that $\mathcal M$ is an $\lll$-antichain, i.e. no two elements lie on the same Dyck path.

Let the triangle $U$ be defined by $U_{i,j}=1$ when $(i,j)\in\mathcal M$ and $U_{i,j}=0$ otherwise. From the previous paragraph we see that $U\in\Pi_\varepsilon$. We are left to show that $T'=T-U\in\Pi_{\lambda'}$. This is done by checking that we have $S(T',d)\le M(\la',d)$ for every Dyck path $d$ with $S(T,d)=M(\la,d)$ and $M(\la',d)<M(\la,d)$. The latter means that $d$ starts in one of $(1,2),\ldots,(l,l+1)$ and ends either in one of $(l+1,l+2),\ldots,(n-1,n)$ or anywhere in the rightmost vertical column. In other words, we are to show that every such Dyck path meets $\mathcal M$. 

Indeed, let $d$ be a Dyck path with the above properties, it meets $\mathcal T$ since $$S(T,d)=M(\la,d)>M(\la'd)\ge 0.$$ Let $(i',j')$ be the $\lll$-minimal element in $d\cap\mathcal T$, we claim that $(i',j')\in\mathcal M$. Indeed, suppose that there exists a $(i'',j'')\in\mathcal T$ with $(i'',j'')\lll(i',j')$. This means that we can define a Dyck path $d'$ passing through $(i'',j'')$ and $(i',j')$ and coinciding with $d$ to the right of $(i',j')$. We would then have $M(\la,d')=M(\la,d)$ but $S(T,d')>S(T,d)=M(\la,d)$, a contradiction. We are left to check that $i'\le l$. Indeed, suppose that $i'>l$ and let $d''$ start with $$(i',i'+1),(i',i'+2),\ldots,(i',j')$$ (i.e. going down and to the right from $(i',i'+1)$ to $(i',j')$) and coincide with $d$ to the right of $(i',j')$. On one hand, we have $S(T,d'')\ge S(T,d)$, on the other, since $i'>l$ and $a_l>0$, we must have $\lambda_{i'}<\la_l$ and, consequently, $M(\la,d'')<M(\la,d)=S(T,d)$.
\end{proof}

We are now ready to carry out the induction step.
\begin{proof}[Proof of Theorem~\ref{ind}]
Let us define $\la'$ and $\varepsilon$ as in Lemma~\ref{mink}. In view of Proposition~\ref{base} and the principle of mathematical induction we may assume that the statement of Theorem~\ref{ind} applies to $L_{\la'}$ and $L_\varepsilon$.

Consider the module $W=L_{\la'}\otimes L_\varepsilon$. In view of our induction base and our induction hypothesis, we have the basis $\{e_U=\exp(U)u_0,U\in\Pi_{\varepsilon}\}$ in $L_{\varepsilon}$ and the basis $\{e_{T'}=\exp(T')v_0',T'\in\Pi_{\la'}\}$ in $L_{\la'}$, where $u_0$ and $v'_0$ are, respectively, the highest vectors in $L_\varepsilon$ and $L_{\la'}$. This gives us the basis 
\begin{equation}\label{basis}
\{e_U\otimes e_{T'},U\in\Pi_{\varepsilon},T'\in\Pi_{\la'}\} 
\end{equation}
in $W$.

First we show that the set of vectors $\{\exp(T)v_0,T\in\Pi_\la\}$ is linearly independent. Indeed, suppose that we have 
\begin{equation}\label{lindep}
\sum_{T\in\Pi_\la}c_T\exp(T)v_0=0 
\end{equation}
with at least one nonzero $c_T$. Consider the $\prec$-maximal $\exp(T_0)$ with $c_{T_0}\neq 0$. By Lemma~\ref{mink} we may specify $T'_0\in\Pi_{\la'}$ and $U_0\in\Pi_\varepsilon$ with $T'_0+U_0=T_0$. We achieve a contradiction by showing that $e_{U_0}\otimes e_{T'_0}$ must have a nonzero coefficient when the left-hand side of~(\ref{lindep}) is decomposed in basis~(\ref{basis}).

First we show that the vector $e_{U_0}\otimes e_{T'_0}$ appears with coefficient 1 when $\exp(T_0)v_0$ is decomposed in basis~(\ref{basis}). Indeed, $$\exp(T_0)v_0=\sum_{U+T'=T_0}(\exp(U)u_0)\otimes(\exp(T')v'_0)$$ with the sum being taken over all decompositions of $T_0$ into a sum of two number triangles with nonnegative integer elements. In particular, one of the summands in this sum will be $e_{U_0}\otimes e_{T'_0}$. 

Let us consider any other summand $(\exp(U)u_0)\otimes(\exp(T')v'_0)$ and decompose it in basis~(\ref{basis}). If $U\in\Pi_\varepsilon$, then all vectors with nonzero coefficients in this last decomposition are of the form $e_U\otimes\cdot\neq e_{U_0}\otimes e_{T'_0}$. Similarly when $T'\in\Pi_{\la'}$. If both $U\notin\Pi_\varepsilon$ and $T'\notin\Pi_{\la'}$, then, due to our induction hypothesis, all vectors appearing with nonzero coefficients in the decomposition of $(\exp(U)u_0)\otimes(\exp(T')v'_0)$ are of the form $e_{U_1}\otimes e_{T'_1}$ with $\exp(U_1)\prec\exp(U)$ and $\exp(T'_1)\prec\exp(T')$. This implies that $\exp(U_1+T'_1)\prec\exp(U+T')=\exp(T_0)$ and, therefore, $e_{U_1}\otimes e_{T'_1}\neq e_{U_0}\otimes e_{T'_0}$. We have thus shown that $e_{U_0}\otimes e_{T'_0}$ appears with coefficient 0 in the decomposition of $(\exp(U)u_0)\otimes(\exp(T')v'_0)$ and established the claim at the beginning of the previous paragraph.

Now consider any other nonzero $c_T$. We must have $\exp(T)\prec\exp(T_0)$. We again write $$\exp(T)v_0=\sum_{U+T'=T_0}(\exp(U)u_0)\otimes(\exp(T')v'_0).$$ When decomposing any $(\exp(U)u_0)\otimes(\exp(T')v'_0)$ from the right-hand side above in basis~(\ref{basis}) we only obtain nonzero coefficients at vectors $e_{U_1}\otimes e_{T'_1}$ with $\exp(U_1)\preceq\exp(U)$ and $\exp(T'_1)\preceq\exp(T')$, due to our induction hypothesis. This, however, implies that $$\exp(U_1+T'_1)\prec\exp(U+T')=\exp(T)\prec\exp(T_0)$$ and $e_{U_1}\otimes e_{T'_1}\neq e_{U_0}\otimes e_{T'_0}$.

We have shown that $e_{U_0}\otimes e_{T'_0}$ would appear with coefficient $c_{T_0}$ in the left-hand side of~(\ref{lindep}). Therefore, the vectors $\exp(T)v_0,T\in\Pi_\la$ are linearly independent and, due to Corollary~\ref{gt}, comprise a basis in $L_\la$. However, the assertion being made in Theorem~\ref{ind} and being proved inductively is stronger than that and is yet to be established.

Consider some $M$ with $\log M\notin\Pi_\la$ and write 
\begin{equation}\label{decomp}
Mv_0=\sum_{T\in\Pi_\la}k_T\exp(T)v_0
\end{equation}
decomposing it in the newly acquired basis. Similarly to the above argument consider the $\prec$-maximal $T_0$ with $k_{T_0}\neq 0$. We are to show that $\exp(T_0)\prec M$.

We again choose $T'_0\in\Pi_{\la'}$ and $U_0\in\Pi_\varepsilon$ with $T'_0+U_0=T_0$. We decompose both sides of~(\ref{decomp}) in basis~(\ref{basis}) and observe that $e_{U_0}\otimes e_{T'_0}$ appears with a nonzero coefficient in the right-hand side's decomposition. This is proved in the same exact manner as we employed above.

However, the left-hand side is equal to $$\sum_{U+T'=\log M}(\exp(U)u_0)\otimes(\exp(T')v'_0).$$ When a summand in this last sum is decomposed in basis~(\ref{basis}) we only obtain nonzero coefficients at vectors $e_{U_1}\otimes e_{T'_1}$ with $\exp(U_1)\preceq\exp(U)$ and $\exp(T'_1)\preceq\exp(T')$ and, therefore, with $\exp(U_1+T'_1)\preceq M$. This means that we must have $\exp(T_0)\preceq M$, however, $\exp(T_0)=M$ is impossible since $\log M\notin\Pi_\la$. The induction step is completed.
\end{proof}

We have proved Theorem~\ref{ind} and, in view of the discussion in Section~\ref{ordmon}, our main Theorem~\ref{main}.

\section{Compatible PBW degenerations}

To recall what we touched upon in the Introduction and Remark~\ref{remAC}, the FFLV bases for types $A$ and $C$ constructed in~\cite{FFL1} and~\cite{FFL2} have the important property of being bases regardless of the way we order the root vector factors comprising each monomial. For types $A$ and $C$ this property is a consequence of the fact that the FFLV bases induce a basis in the abelian PBW degeneration of the representation. In short, to define the abelian PBW degeneration of a highest weight module one views it as module over the corresponding nilpotent subalgebra. The filtration of the universal enveloping algebra of this nilpotent subalgebra by PBW degree induces a filtration of the module, the associated graded space is a module over the symmetric algebra of the nilpotent subalgebra. We refer to this last object as the {\it abelian PBW degeneration} of the initial module. See~\cite{FFL1} and~\cite{FFL2} for details.

Since the root vectors comprising a monomial may not be ordered arbitrarily for the bases constructed here, these bases do not induce bases in the abelian PBW degenerations. The goal of this section is to show that they, instead, induce a basis in a different associated graded space which we now define. (We still, however, refer to this space as a {\it PBW degeneration} of the initial module.)

First, we define a $\frac12\mathbb Z_{\ge 0}$-filtration on the universal enveloping algebra $\mathcal U(\mathfrak n^-)$. This is done simply by defining the filtration element $\mathcal U(\mathfrak n^-)_m$, $m\in\frac12\mathbb Z_{\ge 0}$ as the space spanned by ordered monomials $M$ with $\grad M\le m$. The fact that we thus indeed obtain a filtered algebra, i.e. that for $M$ and $N$ ordered we have $MN\in\mathcal U(\mathfrak n^-)_{\grad M+\grad N}$, follows from Lemma~\ref{reorder}. 

Let us denote the associated $\frac12\mathbb Z_{\ge 0}$-graded algebra $\Phi$. It is the associative algebra generated by $n^2$ elements $\varphi_{i,j}$, $1\le i<j\le 2n+1-i$ with the relations $[\varphi_{i,j},\varphi_{k,l}]=0$ when either $i+j<2n+1$ or $k+l<2n+1$ and $[\varphi_{i,\overline i},\varphi_{k,\overline k}]=2\varphi_{i,\overline k}$ when $i<k$.

We next define a $\frac12\mathbb Z_{\ge 0}$-filtration on $L_\la$ by setting $(L_\la)_m=\mathcal U(\mathfrak n^-)_m(v_0)$. The associated graded space $R_\lambda$ is naturally a $\Phi$-module. In $R_\lambda$ we can choose a ``highest vector'' $w_0$, an element of the one-dimensional component of grading 0.

For a number triangle $T=\{T_{i,j},1\le i<j\le\overline i\}$ with nonnegative integer elements denote $\exp^*(T)\in\Phi$ the monomial in the elements $\varphi_{i,j}$ containing each $\varphi_{i,j}$ in total degree $T_{i,j}$ and, moreover, not containing an $\varphi_{i,\overline i}$ to the right of an $\varphi_{j,\overline j}$ with $i<j$ (``arranged'' with respect to the $\varphi_{i,j}$). The monomial $\exp^*(T)$ does not depend on the positions in which the $\varphi_{i,j}$ with $i+j<2n+1$ are found due to the above commutation relations in $\Phi$.
\begin{theorem}\label{degen}
The set of vectors $\{\exp^*(T)w_0,T\in\Pi_\la\}$ constitutes a basis in $R_\la$.
\end{theorem}
\begin{proof}
For $m\in\frac12\mathbb Z_{\ge 0}$ denote $(R_\la)_m$ the component of $R_\la$ of grading $m$. Consider an ordered monomial $M\in\mathcal U(\mathfrak n^-)$ with $\grad M=m$. We have $Mv_0\in(L_\la)_m$ and the composition map 
\begin{equation}\label{compos}
(L_\la)_m\to(L_\la)_m/(L_\la)_{m-\frac12}\to(R_\la)_m 
\end{equation}
sends $M v_0$ to $\exp^*(\log M)w_0$ (we set $(L_\la)_{-\frac12}=\{0\}$). 

Since the relation $\prec$ respects the grading $\grad$, from Theorem~\ref{ind} and Theorem~\ref{main} it follows that for any $p\in\frac12\mathbb Z_{\ge 0}$ the set $\{Kv_0,\log K\in\Pi_\la,\grad K\le p\}$ constitutes a basis in $(L_\la)_p$. Consequently, the images under the left (surjective) map in~(\ref{compos}) of the vectors $Kv_0$ with $\log K\in\Pi_\la$ and $\grad K=m$ comprise a basis in $(L_\la)_m/(L_\la)_{m-\frac12}$. The right map is bijective which completes the proof.
\end{proof}

Before adding two final remarks, let us observe that Theorem~\ref{main} can, in fact, be easily deduced from Theorem~\ref{degen}. Indeed, if we choose an arranged monomial $M_T$ for every $T\in\Pi_\la$ as in Theorem~\ref{main}, then Theorem~\ref{degen} shows that the images of the vectors $\{M_Tv_0,T\in\Pi_\la,\grad M_T=m\}$ comprise a basis in $(L_\la)_m/(L_\la)_{m-\frac12}$ and Theorem~\ref{main} follows.

\begin{remark}
One could point out more specifically where our argument would break down if we had $\grad f_{i,\overline i}=1$ as one would in the abelian case. Within our proof of the induction base in several cases we would be replacing the product of ``problematic'' elements with something of a greater $\grad$-grading. (See page~\pageref{outline} for the general outline of the case-by-case argument in the proof of Proposition~\ref{base}.)
\end{remark}

\begin{remark}
The abelian PBW degenerations discussed above let one define the corresponding abelian degenerations of flag varieties as is done in~\cite{Fe2}. We point out that a similar definition can be provided for the modified PBW filtration introduced above. 

First, note that the elements $\varphi_{i,j}\in\Phi$ span a Lie algebra with respect to the induced commutation relations, we denote this Lie algebra $\mathfrak q$. The algebra $\Phi$ is then naturally the universal enveloping algebra of $\mathfrak q$. 

We now may consider the connected simply connected Lie group $Q$ with $\mathrm{Lie}(Q)=\mathfrak q$, the group $Q$ acts on the graded space $R_\la$ and on its projectivization $\mathbb{P}(R_\la)$. We may consider the point $w'_0\in\mathbb{P}(R_\la)$ corresponding to $\mathbb Cw_0$ and define a certain analog of the degenerate flag variety as the closure $X=\overline{Qw'_0}$.

In view of the results on the abelian degenerations,  a natural question to ask about $X$ is, for instance, whether it provides a flat degeneration of the corresponding type $B$ flag variety.

Another natural geometric object to consider is the toric variety associated with polytope $P_\la$ (see Remark~\ref{polytope}). Does this variety provide a flat degeneration of the type $B$ flag variety in analogy with FFLV polytopes in types $A$ and $C$?
\end{remark}

\section*{Acknowledgements}
The author would like to thank Evgeny Feigin for numerous helpful discussions of these topics. On top of that, Evgeny was the one to propose the grading $\grad$ by conjecturing Theorem~\ref{degen}. The author would also like to thank the Max Planck Institute for Mathematics for providing a welcoming atmosphere, which stimulated the work on these subjects.

The study was partially supported by the Russian Academic Excellence Project ``5-100''.

\end{document}